\documentclass[12pt]{article}
\usepackage{amsmath,amssymb,amsthm}

\usepackage{graphicx}
\usepackage{epsfig}
\usepackage{lmodern}
\usepackage{amscd}
\usepackage{latexsym}
\usepackage[numbers]{natbib}
\usepackage{tabularx}
\usepackage{a4wide}
\usepackage[usenames]{color}
\usepackage{tikz}
\usetikzlibrary{shapes,decorations.pathreplacing,calc,positioning}
\usepackage{enumitem}
\usepackage{subfig,mathtools}
\usepackage{microtype}
\usepackage{url,booktabs,adjustbox}
\usepackage[font={small}]{caption}

\newcommand{\abs}[1]{\left \lvert #1 \right \rvert}
\newcommand{\sst}[2]{\left\{#1\,:\,#2\right\}}
\newcommand{\ceil}[1]{\left \lceil #1 \right \rceil}
\newcommand{\floor}[1]{\left \lfloor #1 \right \rfloor}
\renewcommand{\le}{\leqslant} %
\renewcommand{\ge}{\geqslant} %
\renewcommand{\leq}{\leqslant} %
\renewcommand{\geq}{\geqslant} %

\DeclareMathOperator{\dist}{dist}

\newtheorem{thm}{Theorem}
\newtheorem{conjecture}[thm]{Conjecture}
\newtheorem{lem}[thm]{Lemma}
\newtheorem{defn}[thm]{Definition}

\title{Closeness Centralization Measure for Two-mode Data of Prescribed Sizes\thanks{This work was supported by PHC Proteus 26818PC, Slovenian ARRS bilateral projects BI-FR/12-13-PROTEUS-011 and BI-FR/14-15-PROTEUS-001.}}
\date{\today}
\author{Matja\v{z} Krnc\footnote{Faculty of Mathematics, Natural Sciences and Information Technologies, University of Primorska, Slovenia.
\texttt{matjaz.krnc@gmail.com}.}
  \and
       Jean-S\'ebastien Sereni\footnote{CNRS (LORIA), Vand\oe uvre-l\`es-Nancy, France.
\texttt{sereni@kam.mff.cuni.cz}. This author's work was partially
supported by the French \emph{Agence Nationale de la Recherche} under reference
\textsc{anr 10 jcjc 0204 01}. \textbf{Corresponding author,
$+33\,354\,958\,640$}.}
  \and
       Riste \v{S}krekovski\footnote{Department of Mathematics, University of Ljubljana, 
and Faculty of information studies, Novo Mesto, and FAMNIT, University of
Primorska, Koper, Slovenia.  Partially supported by ARRS Program P1-0383. 
Email: \texttt{skrekovski@gmail.com}}
  \and
       Zelealem B. Yilma\footnote{Carnegie Mellon University Qatar, Doha, Qatar.
E-mail: \texttt{zyilma@qatar.cmu.edu}. This author's work was partially supported by the French \emph{Agence Nationale de la Recherche} under reference \textsc{anr 10 jcjc 0204 01}.}}

\newlength{\defaultpgflinewidth}
\begin{document}
\pagenumbering{gobble}
\clearpage
\maketitle

\vspace{1.5cm}

\newpage
\pagenumbering{arabic}
\clearpage

\vspace{.8cm}
\begin{center}
\LARGE{\textbf{Closeness Centralization Measure for Two-mode Data of
Prescribed Sizes}}
\end{center}

\vspace{1cm}

\abstract{We confirm a conjecture
by Everett, Sinclair, and Dankelmann~[Some Centrality results new
and old, J. Math. Sociology 28 (2004), 215--227] regarding
the problem of maximizing closeness centralization in two-mode data, where
the number of data of each type is fixed.
Intuitively, our result states that among all networks
obtainable via two-mode data, the largest closeness is achieved
by simply locally maximizing the closeness of a node.
Mathematically, our study concerns
bipartite graphs with fixed size bipartitions, and
we show that the extremal configuration is a rooted tree
of depth~$2$, where neighbors of the root have an equal or
almost equal number of children.}

\vspace{1cm}

\section{Introduction}
A social network is often conveniently modeled by a graph: nodes represent
individual persons and edges represent the relationships between pairs of
individuals.  Our work focuses on simple unweighted graphs: our graph only
tells us, for a given (binary) relation~$R$, which pairs of individual are in
relation according to~$R$.

Centrality is a crucial concept in studying social
networks~\citep*{freeman_1979_centrality, c11}.  It can be seen as a measure of
how central is the position of an individual in a social network.  Various
node-based measures of the centrality have been proposed to determine the
relative importance of a node within a graph (the reader is referred to the
work of~\citet{C8} for an overview).  Some widely used centrality measures are
the degree centrality, the betweenness centrality, the closeness centrality and
the eigenvector centrality (definitions and extended discussions are found in
the book edited by~\citet{c4}).

We focus on closeness centrality, which measures how close a node is to all
other nodes in the graph: the smaller the total distance from a node~$v$ to all
other nodes, the more important the node~$v$ is.  Various closeness-based
measures have been developed~\citep{Bav50, Bea65, BRS92, Nie73, MoMo74, Sab66,
VaFo98, Nie73}.

Let us see an example: suppose we want to place a service facility, e.g.,
a school, such that the total distance to all inhabitants in the region is
minimal. This would make the chosen location as convenient as possible for most
inhabitants.  In social network analysis the centrality index based on this
concept is called \emph{closeness centrality}.

Formally, for a node~$v$ of a graph~$G$, the \emph{closeness} of~$v$ is
defined to be
\begin{equation}
\label{closeness}
C_G(v)  \coloneqq \frac{1}{\sum_{u \in V(G)} \dist_G(v,u)},
\end{equation}
where $\dist_G(u,v)$ is the \emph{distance} between~$u$ and~$v$ in~$G$, that
is, the length of a shortest path in~$G$ between nodes~$u$ and~$v$.  We shall
use the shorthand $W_G(v) \coloneqq \sum_{u \in V(G)} d(v,u)$. In both
notations, we may drop the subscript when there is no risk of confusion.

While centrality measures compare the importance of a node within a graph, the
associated notion of \emph{centralization}, as introduced
by~\citet{freeman_1979_centrality}, allows us to compare the relative
importance of nodes within their respective graphs.  The closeness
centralization of a node $v$ in a graph $G$ is given by
\begin{equation}
\label{centralization}
C_1(v;G) \coloneqq \sum_{u \in V(G)} \big[C(v) - C(u)\big] .
\end{equation}
Further, we set $C_1(G)\coloneqq\max\sst{C_1(v;G)}{v\in V(G)}$.

It is important to note that the parameter~$C_1$ is really tailored to compare
the centralization of nodes in different graphs. If only one graph is
involved, then one readily sees that maximizing~$C_1(v;G)$ over the nodes of
a graph~$G$ amounts to minimizing~$W_G$. Indeed, suppose that $G$ is a graph
and~$v$ a node of~$G$ such that $W_G(v)\le W_G(u)$ for every $u\in V(G)$.
Then for every node~$x$ of~$G$,
\begin{align*}
C_1(v;G)-C_1(x;G)&=(n-1)\left(\frac{1}{W_G(v)}-\frac{1}{W_G(x)}\right)-\left(\frac{1}{W_G(x)}-\frac{1}{W_G(v)}\right)\\
&=n\left(\frac{1}{W_G(v)}-\frac{1}{W_G(x)}\right)\\
&\ge0.
\end{align*}

In what follows, we use the the following notation.
The \emph{star graph} of order~$n$,
sometimes simply known as an \emph{$n$-star}, is the tree on~$n+1$~nodes with one
node having degree~$n$. The star graph is thus a complete bipartite graph with
one part of size~$1$.
\citet*{ESD04} established that over all graphs with a fixed number of nodes,
the closeness is maximized by the star graph.
\begin{thm}\label{thm-all}
If $G$ is a graph with~$n$ nodes, then
\[C_1(u; S_{n-1}) \geq C_1(G),\]
where~$u$ is the node of~$S_{n-1}$ of maximum degree.
\end{thm}

\begin{figure}[t!]
\adjustbox{valign=M}{\begin{minipage}[b]{0.45\linewidth}
\begin{center}
\begin{tikzpicture}[thick,scale=0.8,nvertex/.style={circle, draw=black, fill=white, inner sep=0.5pt, minimum
        size=5pt}]
    \foreach \x [evaluate = \x as \y using int(1+\x/2)] in {-2,0,2}
{
  \draw (\x cm, -2) node[nvertex](s\y){};
  \node[above] at (s\y) {$S_{\y}\strut$};
}
    \foreach \x [evaluate = \x as \y using int(2+\x)] in {-1.5,-.5,.5,1.5}
    {
     \draw  (1.25*\x cm,-4) node[nvertex](c\y){};
     \node[below] at (c\y) {$L_{\y}\strut$};
    };
    
\draw (c0)--(s0)--(c1);
\draw (s0)--(c2)--(s1)--(c3)--(s2)--(c2);
\end{tikzpicture}
\end{center}
\end{minipage}}
\hspace{0.5cm}
\adjustbox{valign=M}{\begin{minipage}[b]{0.45\linewidth}
\begin{center}
\begin{tabular}{lcc}
   \toprule
$v\in V(N)$ & $C_N(v)$ & $C_1(v,N)$  \\
   \midrule
$S_0$  & $1/10$ & $\phantom{-}0.1222$ \\
$S_1$  & $1/12$ & $\phantom{-}0.0055$ \\
$S_2$  & $1/12$ & $\phantom{-}0.0055$ \\
\midrule
$L_0$  & $1/15$ & $-0.1111$ \\
$L_1$  & $1/15$ & $-0.1111$ \\
$L_2$  & $1/9$ & $\phantom{-}0.2000$ \\
$L_3$  & $1/15$ & $-0.1111$ \\
   \bottomrule
\end{tabular}
\end{center}
\end{minipage}}
\par
\protect\caption{\label{fig:new}A two-mode network~$N$ with $7$ nodes ($3$ in
one part, $4$ in the other) and $7$ edges, with the corresponding values for
$C_N$ and $C_1$.}
\end{figure}

They also considered the problem of maximizing centralization measures for
two-mode data~\cite*{ESD04}. In this context, the relation studied links two
different types of data (e.g.,~persons and events) and we are interested in the
centralization of one type of data only (e.g.,~the most central person). Thus
the graph obtained is \emph{bipartite}: its nodes can be partitioned into two
parts so that all the edges join nodes belonging to different parts.  A toy
example is depicted in Figure~\ref{fig:new}, where one type of data consists of
students and the other of classes: edges link the students to the classes they
attended. (The sole  purpose of this example is to make sure the reader is at
ease with the definitions of~$C$ and~$C_1$.) Closeness centrality is maximized
at the student~``$S_0$'' for one part and at the class~``$L_2$'' for the other.
An example of a real-world two-mode network~$N$ on $89$~edges with partition
sizes $\abs{P_{1}}=18$ and $\abs{P_{2}}=14$, borrowed from \cite{davis1969deep}
is depicted on Figure~\ref{fig:2mode-netw:Circles-squares}.  On the figure, one
can observe a frequency of interparticipation of a group of women in social
events in Old City,~1936.  On
Tables~\ref{tab:closeness-2mode-faculties}~and~\ref{tab:closeness-2mode-courses},
one can observe closeness centralization for partitions~$P_{1}$ and~$P_{2}$ and
notice that closeness centrality (and hence centralization) is maximized at
``Mrs.~Evelyn~Jefferson'' and the event from ``September~16th'', respectively.

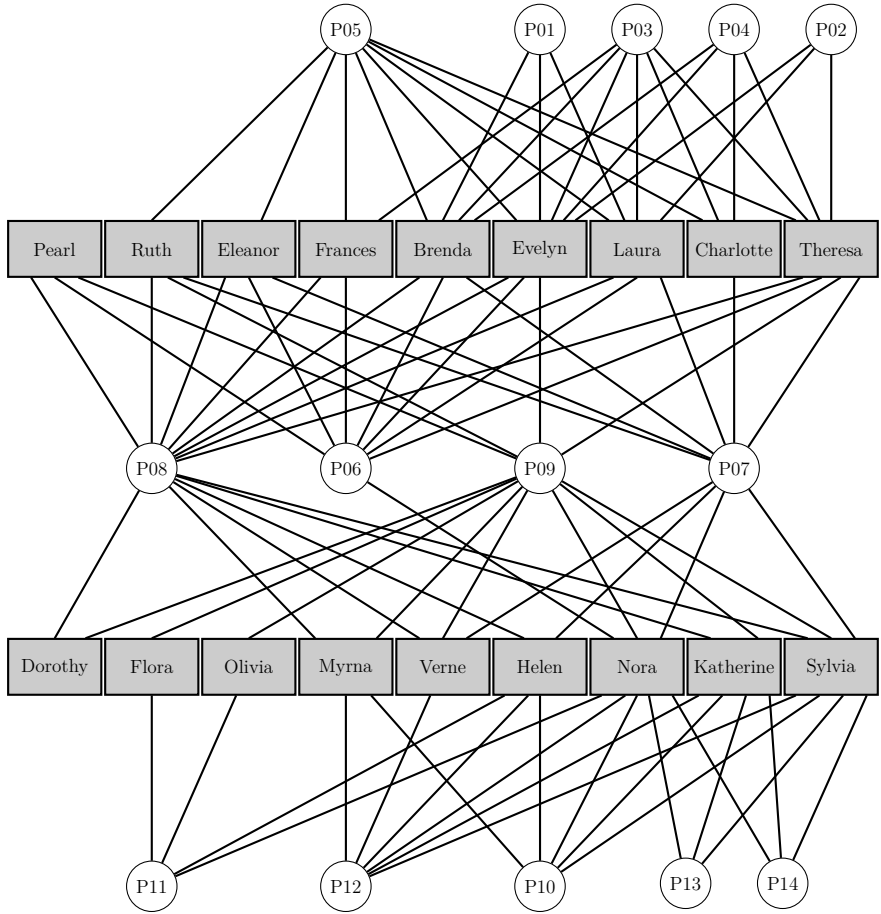
\begin{figure}[t!]
\begin{centering}
\begin{tikzpicture}[node distance=2.15cm,auto,thick,every node/.style={scale=0.6},scale=0.8]
\tikzstyle{lady} = [draw, thick, fill=black!20, minimum height=3em,minimum width=5em]
\tikzstyle{event} = [circle, draw, thin, minimum height=2.5em]
   
\node[lady] (P) {Pearl};
\node[lady, right of=P] (R) {Ruth};
\node[lady, right of=R] (E) {Eleanor};
\node[lady, right of=E] (Fr) {Frances};
\node[lady, right of=Fr] (B) {Brenda};
\node[lady, right of=B] (Ev) {Evelyn};
\node[lady, right of=Ev] (L) {Laura};
\node[lady, right of=L] (C) {Charlotte};
\node[lady, right of=C] (T) {Theresa};

\node[lady, below= 48mm of P] (D) {Dorothy};
\node[lady, below= 48mm of R] (F) {Flora};
\node[lady, below= 48mm of E] (O) {Olivia};
\node[lady, below= 48mm of Fr] (M) {Myrna};
\node[lady, below= 48mm of B] (V) {Verne};
\node[lady, below= 48mm of Ev] (H) {Helen};
\node[lady, below= 48mm of L] (N) {Nora};
\node[lady, below= 48mm of C] (K) {Katherine};
\node[lady, below= 48mm of T] (S) {Sylvia};

\node[event, above=22mm of Fr] (5) {P05};
\node[event, above=22mm of Ev] (1) {P01};
\node[event, above=22mm of L] (3) {P03};
\node[event, above=22mm of C] (4) {P04};
\node[event, above=22mm of T] (2) {P02};

\draw (5)--(R.north);
\draw (5)--(E);
\draw (5)--(Fr);
\coordinate (x) at ($(B.north west)!.66!(B.north)$);
\draw (5)--(x);
\coordinate (x) at ($(Ev.north west)!.5!(Ev.north)$);
\draw (5)--(x);
\coordinate (x) at ($(L.north west)!.4!(L.north)$);
\draw (5)--(x);
\coordinate (x) at ($(C.north west)!.33!(C.north)$);
\draw (5)--(x);
\coordinate (x) at ($(T.north west)!.25!(T.north)$);
\draw (5)--(x);

\draw (1)--(B.north);
\coordinate (x) at ($(Ev.north west)!1!(Ev.north)$);
\draw (1)--(x);
\draw (1)--(L);

\coordinate (x) at ($(Fr.north east)!.3!(Fr.north)$);
\draw (3)--(x);
\coordinate (x) at ($(B.north east)!.66!(B.north)$);
\draw (3)--(x);
\coordinate (x) at ($(Ev.north east)!.75!(Ev.north)$);
\draw (3)--(x);
\draw (3)--(L);
\coordinate (x) at ($(C.north west)!.66!(C.north)$);
\draw (3)--(x);
\coordinate (x) at ($(T.north west)!.5!(T.north)$);
\draw (3)--(x);

\coordinate (x) at ($(T.north west)!.33!(T.north)$);
\draw (2)--(T);
\coordinate (x) at ($(Ev.north east)!.25!(Ev.north)$);
\draw (2)--(x);
\coordinate (x) at ($(L.north east)!.5!(L.north)$);
\draw (2)--(x);

\coordinate (x) at ($(T.north west)!.75!(T.north)$);
\draw (4)--(x);
\coordinate (x) at ($(Ev.north east)!.5!(Ev.north)$);
\draw (4)--(x);
\coordinate (x) at ($(C.north east)!1!(C.north)$);
\draw (4)--(x);
\coordinate (x) at ($(B.north east)!.33!(B.north)$);
\draw (4)--(x);

\node[event, below=22mm of Ev] (9) {P09};
\node[event, below=22mm of R] (8) {P08};
\node[event, below=22mm of Fr] (6) {P06};
\node[event, below=22mm of C] (7) {P07};

\coordinate (x) at ($(T.south east)!.4!(T.south west)$);
\draw (9)--(x);
\draw (9)--(Ev);
\coordinate (x) at ($(P.south east)!.5!(P.south)$);
\draw (9)--(x);
\coordinate (x) at ($(R.south east)!.66!(R.south)$);
\draw (9)--(x);

\coordinate (x) at ($(D.north east)!.35!(D.north)$);
\draw (9)--(x);
\draw (9)--(F.north);
\draw (9)--(O.north);
\coordinate (x) at ($(M.north east)!.35!(M.north)$);
\draw (9)--(x);
\draw (9)--(V.north);
\coordinate (x) at ($(N.north west)!.5!(N.north east)$);
\draw (9)--(x);
\coordinate (x) at ($(K.north east)!.5!(K.north)$);
\draw (9)--(x);
\draw (9)--(S.north);

\coordinate (x) at ($(T.south west)!.2!(T.south east)$);
\draw (8)--(x);
\coordinate (x) at ($(L.south west)!.5!(L.south)$);
\draw (8)--(x);
\coordinate (x) at ($(B.south west)!.5!(B.south)$);
\draw (8)--(x);
\coordinate (x) at ($(P.south west)!.5!(P.south)$);
\draw (8)--(x);
\draw (8)--(R.south);
\coordinate (x) at ($(E.south west)!.5!(E.south)$);
\draw (8)--(x);
\draw (8)--(Fr);
\coordinate (x) at ($(Ev.south west)!.33!(Ev.south)$);
\draw (8)--(x);

\draw (8)--(D.north);
\coordinate (x) at ($(M.north west)!.35!(M.north)$);
\draw (8)--(x);
\coordinate (x) at ($(V.north west)!.5!(V.north)$);
\draw (8)--(x);
\coordinate (x) at ($(K.north west)!.5!(K.north)$);
\draw (8)--(x);
\coordinate (x) at ($(S.north west)!.5!(S.north)$);
\draw (8)--(x);
\coordinate (x) at ($(H.north west)!.66!(H.north)$);
\draw (8)--(x);

\draw (7)--(C);
\coordinate (x) at ($(L.south east)!.5!(L.south)$);
\draw (7)--(x);
\coordinate (x) at ($(B.south east)!.5!(B.south)$);
\draw (7)--(x);
\coordinate (x) at ($(R.south east)!.33!(R.south)$);
\draw (7)--(x);
\coordinate (x) at ($(E.south east)!.5!(E.south)$);
\draw (7)--(x);
\coordinate (x) at ($(T.south east)!.4!(T.south)$);
\draw (7)--(x);

\coordinate (x) at ($(V.north east)!.5!(V.north)$);
\draw (7)--(x);
\coordinate (x) at ($(H.north east)!.66!(H.north)$);
\draw (7)--(x);
\coordinate (x) at ($(N.north east)!.25!(N.north west)$);
\draw (7)--(x);
\coordinate (x) at ($(S.north east)!.5!(S.north)$);
\draw (7)--(x);

\coordinate (x) at ($(P.south east)!1!(P.south)$);
\draw (6)--(x);
\coordinate (x) at ($(N.north west)!.25!(N.north east)$);
\draw (6)--(x);
\coordinate (x) at ($(L.south west)!1!(L.south)$);
\draw (6)--(x);
\coordinate (x) at ($(B.south west)!1!(B.south)$);
\draw (6)--(x);
\coordinate (x) at ($(T.south west)!.4!(T.south east)$);
\draw (6)--(x);
\coordinate (x) at ($(E.south east)!1!(E.south)$);
\draw (6)--(x);
\coordinate (x) at ($(Fr.south west)!1!(Fr.south)$);
\draw (6)--(x);
\coordinate (x) at ($(Ev.south west)!.66!(Ev.south)$);
\draw (6)--(x);

\node[event, below=22mm of F] (11) {P11};
\node[event, below=22mm of H] (10) {P10};
\node[event, below=22mm of M] (12) {P12};
\node[event,shift={(0,-48mm)}] (13) at ($(N)!0.5!(K)$) {P13};
\node[event,shift={(0,-48mm)}] (14) at ($(S)!0.5!(K)$) {P14};

\draw (11)--(F);
\draw (11)--(O);
\coordinate (x) at ($(H.south west)!.25!(H.south)$);
\draw (11)--(x);
\coordinate (x) at ($(N.south west)!.25!(N.south)$);
\draw (11)--(x);

\draw (10)--(M);
\draw (10)--(H);
\draw (10)--(N.south);
\coordinate (x) at ($(K.south west)!.75!(K.south)$);
\draw (10)--(x);
\coordinate (x) at ($(S.south west)!.75!(S.south)$);
\draw (10)--(x);

\draw (12)--(M);
\draw (12)--(V);
\coordinate (x) at ($(H.south west)!.75!(H.south)$);
\draw (12)--(x);
\coordinate (x) at ($(N.south west)!.75!(N.south)$);
\draw (12)--(x);
\coordinate (x) at ($(K.south west)!.25!(K.south)$);
\draw (12)--(x);
\coordinate (x) at ($(S.south west)!.25!(S.south)$);
\draw (12)--(x);

\coordinate (x) at ($(N.south east)!.75!(N.south)$);
\draw (13)--(x);
\coordinate (x) at ($(K.south east)!.75!(K.south)$);
\draw (13)--(x);
\coordinate (x) at ($(S.south east)!.75!(S.south)$);
\draw (13)--(x);

\coordinate (x) at ($(N.south east)!.25!(N.south)$);
\draw (14)--(x);
\coordinate (x) at ($(K.south east)!.25!(K.south)$);
\draw (14)--(x);
\coordinate (x) at ($(S.south east)!.25!(S.south)$);
\draw (14)--(x);
\end{tikzpicture}
\par\end{centering}
\protect\caption{\label{fig:2mode-netw:Circles-squares}A two-mode network~$N$
on $89$~edges with partition sizes $n_{0}=18$ and $n_{1}=14$.  The network
represents the participation of a given set of people in the social events from
1936 reported in Old City Herald, where circles represent social events while
rectangles represent women (see Tables~\ref{tab:closeness-2mode-faculties}
and~\ref{tab:closeness-2mode-courses}).}
\end{figure}

\begin{table}
\begin{centering}
\begin{tabular}{lcc}
   \toprule
$v\in P_1$ & $C_N(v)$  & $C_1(v,N)$  \\
   \midrule
   Mrs. Evelyn Jefferson  & $0.01667$ & $\phantom{-}0.07779$ \\
   Miss Theresa Anderson  & $0.01667$ & $\phantom{-}0.07779$ \\
   Mrs. Nora Fayette			 & $0.01667$ & $\phantom{-}0.07779$ \\
   Mrs. Sylvia Avondale	 & $0.01613$ & $\phantom{-}0.06058$ \\
   Miss Laura Mandeville	 & $0.01515$ & $\phantom{-}0.02930$ \\
   Miss Brenda Rogers		 & $0.01515$ & $\phantom{-}0.02930$ \\
   Miss Katherine Rogers	 & $0.01515$ & $\phantom{-}0.02930$ \\
   Mrs. Helen Lloyd			 & $0.01515$ & $\phantom{-}0.02930$ \\
   Miss Ruth DeSand			 & $0.01471$ & $\phantom{-}0.01504$ \\
   Miss Verne Sanderson	 & $0.01471$ & $\phantom{-}0.01504$ \\
   Miss Myra Liddell			 & $0.01429$ & $\phantom{-}0.00160$ \\
   Miss Frances Anderson	 & $0.01389$ & $-0.01110$\\
   Miss Eleanor Nye			 & $0.01389$ & $-0.01110$ \\
   Miss Pearl Oglethorpe	 & $0.01389$ & $-0.01110$\\
   Mrs. Dorothy Murchison & $0.01351$ & $-0.02311$\\
   Miss Charlotte McDowd	 & $0.01250$ & $-0.05555$ \\
   Mrs. Olivia Carleton	 & $0.01220$ & $-0.06530$ \\
   Mrs. Flora Price			 & $0.01220$ & $-0.06530$ \\
   \bottomrule
\end{tabular}
\par\end{centering}
\protect\caption{\label{tab:closeness-2mode-faculties}Nodes from the group of women and their closeness values.}
\end{table}

\begin{table}
\begin{centering}
\begin{tabular}{lccc}
   \toprule
$v\in P_2$ & label on Fig.~\ref{fig:2mode-netw:Circles-squares} & $C_N(v)$ & $C_1(v,N)$  \\
   \midrule
   September 16th  & P8  & $0.01923$  & $\phantom{-}0.15984$ \\
   April 8th  & P9       & $0.01786$  & $\phantom{-}0.11588$ \\
   March 15th  & P7      & $0.01667$  & $\phantom{-}0.07779$ \\
   May 19th  & P6        & $0.01562$  & $\phantom{-}0.04445$ \\
February 25th  & P5   & $0.01351$ & $-0.02311$ \\
April 12th  & P3      & $0.01282$ & $-0.04529$ \\
April 7th	 & P12  		& $0.01282$ & $-0.04529$ \\
June 10th	 & P10  		& $0.01250$ & $-0.05555$ \\
September 26th  & P4  & $0.01220$ & $-0.06530$ \\
February 23rd	 & P11  & $0.01220$ & $-0.06530$ \\
June 27th  & P1  			& $0.01190$ & $-0.07459$ \\
March 2nd	 & P2  			& $0.01190$ & $-0.07459$ \\
November 21st	 & P13  & $0.01190$ & $-0.07459$ \\
August 3rd  & P14  		& $0.01190$ & $-0.07459$ \\
   \bottomrule
\end{tabular}
\par\end{centering}

\protect\caption{\label{tab:closeness-2mode-courses}Nodes from the partition of social events from 1936, reported in \emph{Old City Herald}, and their closeness values.}
\end{table}

~\citeauthor{ESD04} formulated an interesting conjecture, which was later proved by
\citet{Sin04}. To state it, we first need a definition.
\begin{defn}
Let $H(u;n_0,n_1)$ be the tree with node bipartition
$(A_0,A_1)$ such that
\begin{itemize}
   \item $\abs{A_i} = n_i$ for $i\in\{0,1\}$;
   \item there exists a node $u \in A_0$ such that $N_G(u) = A_1$; and
   \item $\deg(w) \in \left \{1+ \ceil{\frac{n_0-1}{n_1}},  1 + \floor{\frac{n_0-1}{n_1}} \right \}$
for all nodes $w \in A_1$.
\end{itemize}
The node~$u$ is called \emph{the root} of~$H(u;n_0,n_1)$.
\end{defn}

The aforementioned conjecture was that the pair~$(H(u;n_0,n_1),u)$ is an
\emph{extremal pair} for the problem of maximizing \emph{betweenness
centralization} in bipartite graphs with a fixed sized bipartition into parts
of sizes~$n_0$ and~$n_1$.  Recall that for two-mode data, we are only
interested in one type of data: in graph-theoretic terms, we look only at nodes
that belong to the part of size~$n_0$, and we want to know which of these nodes
has the largest closeness in the graph. In other words, letting~$A_0$ be the
part of size~$n_0$ of~$V(G)$, we want to determine $\max\sst{C_1(v;G)}{v \in
A_0}$.

\citeauthor{ESD04} also suggested that the same pair is extremal for closeness
and eigenvector centralization measures.  In this paper, we confirm the
conjecture for the closeness centralization measure.  That is, we prove that
the pair $H(v;n_0,n_1)$ is extremal for the problem of maximizing closeness
centralization in bipartite graphs with parts of size~$n_0$ and~$n_1$,where~$v$
is the root.

We point out that a similar study for the centrality measure of eccentricity
was led recently~\cite{Krnc2015eccentricity}. In addition, Bell~\cite{Bel14}
worked on closely related notions, namely subgroup centrality measures.
Similarly as for two-mode data, a susbet~$S$ of the nodes is fixed (called
a \emph{group}) and the aim is to find a node in~$S$ with largest centrality.
However, unlike in the standard centrality notion, the centrality itself is
computed using distances only to the nodes in~$S$ (\emph{local centrality}) or
to the nodes outside~$S$ (\emph{global centrality}). Note that the standard
notion, which is used in this work, takes into account the distances to all
other nodes in the graph.

\section{Bipartite Networks With Fixed Number of Nodes}
\begin{thm}
\label{closenessbipartite}
Let $G$ be a bipartite graph with node parts $A_0$ and $A_1$ sizes $n_0$ and $n_1$,
respectively. Then for each $v\in A_0$,
\[C_1(u; H(u;n_0,n_1)) \geq C_1(v;G).\]
\end{thm}

To prove Theorem~\ref{closenessbipartite}, suppose that $G$ is a bipartite graph
with bipartition $(A_0,A_1)$ where $\abs{A_i}=n_i$ for $i\in\{0,1\}$, and
$u$ is a node in~$A_0$ such that $C_1(u;G)\ge C_1(v;H(v;n_0,n_1))$.
We prove that this inequality must actually be an equality
by showing that any such extremal pair $C_1(u;G)$ must satisfy the
following three properties:
\begin{enumerate}[label=(P\arabic*),ref=(P\arabic*),leftmargin=.65in]
\item $G$ is a tree;\label{ptree}
\item $\deg_G(u) = n_1$; and\label{pdegu}
\item $\abs{\deg_G(w_1) - \deg_G(w_2)} \leq 1$ whenever $w_1,w_2 \in A_1$.\label{pdeg1}
\end{enumerate}
Property~\ref{ptree} is relatively straightforward to check and so
is~\ref{pdeg1} if we assume that~\ref{pdegu} holds.
Thus the majority of the discussion below will be devoted to proving
that~\ref{pdegu} holds, which we do last.
For convenience, we define~$V$ to be~$V(G)$.

We start by establishing~\ref{ptree};
namely, that the graph $G$ is a tree.
Assume, for the sake of contradiction, that $G$ is not a tree and
let $T$ be a breadth-first-search tree of~$G$ rooted at~$u$.
Note that $W_G(u) = W_T(u)$ and
$W_T(x) \geq W_G(x)$ for any node $x \in V(G)$.
In addition, there exist at least two nodes
for which the above inequality is strict.
It follows that $C_1(u;T) > C_1(u;G)$, a contradiction.

We now establish that~\ref{pdeg1} holds if~\ref{pdegu} does.
Thus we know that $G$ is a tree and
we assume that $N_G(u) = A_1$, therefore also all
nodes from $A_0\setminus\left\{u\right\}$ are leaves.
Suppose, for the sake of contradiction, that there exist nodes
$w_1, w_2 \in A_1$ such that $\deg(w_1) \geq \deg(w_2)+2$.
Let $z$ be a neighbor of~$w_1$ different from~$u$ and consider the graph~$G'$
obtained by deleting the edge~$w_1z$ and replacing it with~$w_2z$.
Note that $W_{G'}(u) = W_G(u)$ and that
$W_{G'}(x) = W_{G}(x)$ unless $x \in N_G[w_1] \cup N_G[w_2]$,
that is unless~$x$ belongs to the closed neighborhood of either~$w_1$ or~$w_2$.
So
\begin{equation}
\label{distribute}
C_1(u;G') - C_1(u;G) = \sum_{x \in N_G[w_1] \cup N_G[w_2]} \frac{1}{W_G(x)}
\quad - \sum_{x \in N_G[w_1] \cup N_G[w_2]} \frac{1}{W_{G'}(x)} .
\end{equation}
Now, let $N_G(w_1)=\{u,z,x_1, \ldots, x_t\} $ and
$N_G(w_2)=\{u,y_1,\ldots, y_s\}$ where,
by assumption, $t > s$.

Recalling that $G$ is a tree, observe that the following hold for every $i\in\{1,\ldots,t\}$ and
every $j\in\{1,\ldots,s\}$ (for better illustration, see Figure~\ref{fig:subtree}).
\begin{enumerate}[label=(\roman*).]
\item $W_{G'}(x_i) = W_G(x_i) + 2$;
\item $W_{G'}(y_j) = W_G(y_j) - 2$;
\item $W_G(y_j) = W_G(x_i) + 2(t-s+1)>W_G(x_i)+2$;
\item $W_{G'}(z)=W_G(z)+2(t-s)> W_{G}(z)$;
\item $W_{G'}(w_1)=W_{G}(w_1)+2;$ and
\item $W_{G'}(w_2)=W_{G}(w_2)-2$.
\end{enumerate}
From~(i)--(iii), we infer that
for any $j \in \{1, \ldots, s\}$,
\[ \frac{1}{W_{G'}(x_j)} + \frac{1}{W_{G'}(y_j)} < \frac{1}{W_{G}(x_j)} + \frac{1}{W_{G}(y_j)} , \]
and similarly by (v) and (vi),
\[
\frac{1}{W_{G'}(w_1)} + \frac{1}{W_{G'}(w_2)} < \frac{1}{W_{G}(w_1)} + \frac{1}{W_{G}(w_2)}.
\]
Thus the right side of~\eqref{distribute} is greater than
\[
\frac{1}{W_G(z)}-\frac{1}{W_{G'}(z)} + \sum_{j = s+1}^t \frac{1}{W_G(x_j)} - \frac{1}{W_{G'}(x_j)},
\]
which is positive by~(i) and~(iv). This contradiction shows that~\ref{pdeg1} holds
provided~\ref{pdegu} does.

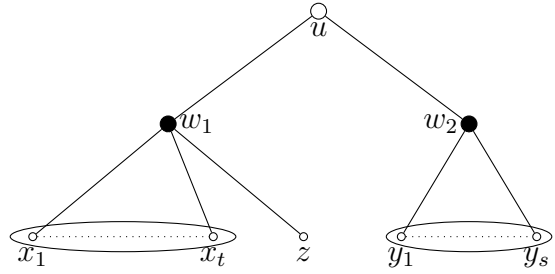
\begin{figure}[t]
\begin{center}
\begin{tikzpicture}[
level 1/.style={sibling distance=4cm},
level 4/.style={sibling distance=1cm},
vertex/.style={circle, draw=black, fill=black, inner sep=0.5pt, minimum
        size=6pt},
wvertex/.style={circle, draw=black, fill=white, inner sep=0.5pt, minimum
        size=6pt},
0-vertex/.style={circle, draw=black, fill=black, inner sep=0mm, minimum
        size=2pt},%
1-vertex/.style={circle, draw=black, fill=white, inner sep=0mm, minimum
        size=3pt},%
2-vertex/.style={circle, draw=black, fill=black, inner sep=0mm, minimum
        size=3pt}%
]
\node[wvertex](u) {}
child {node[vertex](w) {} [sibling distance=6mm]
   child {node[1-vertex](wf) {}}
   child {node {} edge from parent[white]}
   child {node {} edge from parent[white]}
   child {node {} edge from parent[white]}
   child {node[1-vertex](sl) {}}
   child {node {} edge from parent[white]}
   child {node[1-vertex](zzz) {}}
}
child {node[vertex](z) {} [sibling distance=6mm]
   child {node[1-vertex](pf) {}}
   child {node(cp1) {} edge from parent[white]}
   child {node(cp2) {} edge from parent[white]}
   child {node[1-vertex](pl) {}}
}
;
   \draw[dotted] (wf)--(sl) node[midway] (cw){};
   \draw[dotted] (pf)--(pl) node[midway] (cp){};
   \node[left] at (z) {$w_2$};
   \node[right] at (w) {$w_1$};
   \node[below] at (u) {$u$};
   \node[below] at (wf) {$x_1$};
   \node[below] at (sl) {$x_t$};
   \node[below] at (zzz) {$z$};
   \node[below] at (pf) {$y_1$};
   \node[below] at (pl) {$y_s$};

   \draw (cw) ellipse (1.5cm and 2mm);
   \draw (cp) ellipse (1.1cm and 2mm);
\end{tikzpicture}
\end{center}
\caption{The subtree of~$G$ induced by $N_G[w_1]\cup N_G[w_2]$.}\label{fig:subtree}
\end{figure}

It remains to prove that~\ref{pdegu} holds to complete the proof. First, if
$n_1=1$, then the tree~$G$ must be an $n_0$-star, hence the second property is
satisfied. Now consider the case where $n_1=2$. Then there is precisely one
node~$x$ that is adjacent to both nodes in~$A_1$. Moreover, $W_G(x)\le W_G(w)$
if $w\in A_0$ since, if $w\in A_0\setminus\{x\}$ then $W_G(w)\ge
2(n_0-1)+4=2n_0+2$ while $W_G(x)=2+2(n_0-1)=2n_0+1$. Thus $u=x$ and hence
$\deg_G(u)=n_1=2$, as wanted.

From now on, we assume that $n_1\ge3$.  As in the proof of~\ref{pdeg1}, we
argue that if~\ref{pdegu} does not hold then $C_1(u;G)$ can be increased by
altering the graph~$G$.  In this case, however, we find it necessary to use our
assumption that $C_1(u;G)$ itself is at least as large as
$C_1(v;H(v;n_0,n_1))$. This shall allow us to have a lower bound on $C_1(u;G)$,
by the next lemma.
\begin{lem}\label{lem-lowerbound}
$C_1(u;H(u;n_0,n_1)) \geq \frac{n_1-1}{2(2n_1-1)}$.
\end{lem}
\begin{proof}
We establish the inequality via a direct computation.
Unfortunately, the expressions involved force a lengthy computation.

We set $m \coloneqq n_0-1$ and we write $m= p n_1 + r$ where $0 \leq r < n_1$.
Let us now calculate $W(x)$ for each node~$x$ of~$H(u;n_0,n_1)$.

\begin{enumerate}
\item $W(u) = n_1 + 2 m$.
\item Consider the neighbors of~$u$: there are
\begin{enumerate}
\item $r$ neighbors~$x$ for which $W(x) = \ceil{m/n_1} + 1 + 2(n_1-1) +
3(m-\ceil{m/n_1})$; and
\item $n_1-r$ neighbors~$x$ for which $W(x) = \floor{m/n_1} + 1 + 2(n_1-1) + 3(m-\floor{m/n_1})$.
\end{enumerate}
\item Consider the nodes at distance two from~$u$: there are
\begin{enumerate}
\item $r \ceil{m/n_1}$ nodes~$x$ for which $W(x) = 1 + 2 \ceil{m/n_1} +
3(n_1-1) + 4(m-\ceil{m/n_1})$; and
\item $(n_1-r) \floor{m/n_1}$ nodes~$x$ for which $W(x) = 1 + 2 \floor{m/n_1} + 3(n_1-1) + 4(m-\floor{m/n_1})$.
\end{enumerate}
\end{enumerate}

Since $\floor{m/n_1}= (m-r)/n_1$ and, for $r>0$, we have $\ceil{m/n_1} = (m+n_1-r)/n_1$,
it follows that if $r>0$ then
\begin{align}
\begin{split}
\label{C1vdef}
C_1(u)  & =   \frac{n_1+m}{n_1+2m} -  \frac{rn_1}{3mn_1 - 2m + 2n_1^2-3n_1 + 2r}
	- \frac{n_1(n_1-r)}{3mn_1 - 2m + 2n_1^2- n_1 + 2r} \\
	&  -  \frac{r(m+n_1-r)}{4mn_1 - 2m + 3n_1^2-4n_1 + 2r}
	- \frac{(n_1-r)(m-r)}{4mn_1 - 2m +3n_1^2-2n_1 + 2r}
\end{split} \\
\label{C1vlb}
& \geq \frac{n_1+m}{n_1+2m} -  \frac{n_1^2}{3mn_1 - 2m + 2n_1^2-3n_1 + 2r}
	 -  \frac{n_1m}{4mn_1 - 2m + 3n_1^2-4n_1 + 2r},
\end{align}
where we used that $n_1 > 0$ to derive~\eqref{C1vlb}.

\noindent
One notes that~\eqref{C1vlb} is still true if $r=0$. Indeed, in this case
$\lceil\tfrac{m}{n_1}\rceil=\lfloor\tfrac{m}{n_1}\rfloor=\tfrac{m}{n_1}$, so
\[
      C_1(u) = \frac{n_1+m}{n_1+2m} - \frac{n_1^2}{3mn_1-2m+2n_1^2-n1}
            - \frac{n_1m}{4mn_1-2m+3n_1^2-2n_1},
\]
so that~\eqref{C1vlb} stays true.

As is seen from~\eqref{C1vdef}, if $n_1$ is fixed and $n_0$ tends to infinity
(hence, so does $m$), then
$C_1(u)$ approaches $1/2 - n_1/(4n_1-2) = \frac{n_1-1}{4n_1-2}$.

Let us now subtract $\frac{n_1-1}{4n_1-2}$ from the right side of~\eqref{C1vlb}
and show that the difference is non-negative.
After cross-multiplying and simplifying, we obtain a fraction with positive
denominator (since each denominator in the right side of~\eqref{C1vlb} is positive),
and with numerator equal to
\begin{multline}
\label{C1vred}
m^2 (10n_1^4 - 44 n_1^3 + 12n_1^2r + 30n_1^2 - 8n_1r - 4n_1) \\ \hspace{-2cm}
+ m(15n_1^5 - 77n_1^4 + 38 n_1^3r + 74n_1^3 - 54n_1^2r - 14n_1^2 + 8n_1r^2 + 8n_1r) \\
+ (6n_1^6 - 35n_1^5 + 22n_1^4 r + 45n_1^4 - 48n_1^3r - 12n_1^3 + 12n_1^2r^2 + 14n_1^2r - 4n_1r^2).
\end{multline}
This expression increases with~$n_1$ and is clearly positive when $n_1=6$
(to see it quickly just compare, in each parenthesis, every (maximal) sequence of
consecutive negative terms with the (maximal) sequence of positive terms
preceding it).
Further, a direct calculation ensures that~\eqref{C1vred} is actually positive
even when $n_1=5$.

However, if $n_1 \in \{3,4\}$, then~\eqref{C1vred} could take on negative values
for certain values of~$m$.
To deal with these two cases we revert back to the initial equation~\eqref{C1vdef}.

Assume that $n_1=3$.  Then subtracting $\frac{n_1-1}{4n_1-2}$ from both sides
of~\eqref{C1vdef}
yields that $C_1(u)-\frac{n_1-1}{4n_1-2}$ is at least
\begin{equation}
\label{C1k=3}
\frac{m+3}{2m+3} - \frac{3r}{7m + 9 + 2r} - \frac{9-3r}{7m + 15 + 2r}
- \frac{r(m + 3 -r)}{10m + 15 + 2r} - \frac{(3-r)(m-r)}{10m + 21 + 2r} -
\frac{1}{5}.
\end{equation}
Placing~\eqref{C1k=3} under one (positive) denominator,
the numerator becomes
\begin{multline}
\label{C1k3red}
1540 m^4  + 2m^3 (9075 - 1016 r + 588r^2)
+ 6m^2 (10605 - 1047r + 937 r^2 + 112 r^3) \\ \hspace{-2cm}
+ m(88155 - 3816r + 9828 r^2 + 2408 r^3 + 96 r^4)  \\ \hspace{-2cm}
+ (42525 + 1350 r + 6174 r^2 + 2280 r^3 + 184 r^4),
\end{multline}
which is clearly positive as $r \leq n_1-1= 2$.

A similar calculation yields the conclusion when $n_1=4$.
In this case,
the difference of~\eqref{C1vdef} and $\frac{n_1-1}{4n_1-2}$ yields
that
$C_1(u)-\frac{n_1-1}{4n_1-2}$ is at least
\begin{multline*}
\frac{m+4}{2m+4} - \frac{2r}{5m + 10 + r} - \frac{8-2r}{5m + 14 + r}
- \frac{r(m + 4 -r)}{14m + 32 + 2r} - \frac{(4-r)(m-r)}{14m + 40 + 2r} - \frac{3}{14},
\end{multline*}
whose numerator, when placed under a common (positive) denominator, is
\begin{multline*}
\label{C1k4red}
1855 m^4 + 4m^3 (5855 - 82r + 100r^2) +
2m^2 (52090 + 206 r + 1405r^2 + 80 r^3) \\
+ 4 m(49180 + 2022r + 1793 r^2 + 194 r^3 + 4r^4) \\
+ 3(44800 + 4080 r + 2204 r^2 + 332 r^3 + 13r^4).
\end{multline*}
This is non-negative as $r \leq n_1-1= 3$.
This concludes the proof.
\end{proof}

It remains to demonstrate that~\ref{pdegu} holds.
To this end, we consider the tree~$G$ to be rooted at~$u$
and, for a node~$x$, we let $T_x$ be the subtree of~$G$ rooted at~$x$.
To avoid unnecessary notation later, let us observe immediately that
if $\deg_G(u)=1$ then~\ref{pdegu} holds. For otherwise, $n_1\ge2$
and there exists a node~$u'$ at distance two from~$u$ such that
$\deg_G(u')\ge2$. As a result, $W_G(u)\ge W_G(u')+\abs{V(T_{u'})}-1>W_G(u')$,
which implies that $C_1(u';G)>C_1(u;G)$, a contradiction.

We also note that if $\dist_G(u,x) \leq 2$ for all $x \in V(G)$,
then~\ref{pdegu} is satisfied.
So assume that there exists some child of~$u$ whose subtree has depth at least~$2$.
Among all such children of~$u$, let $z$ be such that $\abs{V(T_z)}$ is maximum,
that is,
\[
\abs{V(T_z)}=\max\sst{\abs{V(T_v)}}{\text{$v$ child of~$u$ and $T_v$ has depth at
least~$2$}}.
\]
We now give some notations, which are illustrated in Figure~\ref{fig:sets}.
Let $y_1,\ldots,y_t$ be the nodes of~$T_z$ with depth~$2$
and set $Y\coloneqq\cup_{i=1}^{t}V(T_{y_i})$. Note that, by definition,
${t\ge1}$ and ${\dist_G(u,y_i)=3}$ whenever $1\le i\le t$.
Let $p_1,\ldots,p_{\ell}$ be the children of~$z$ (in~$T_z$) with degree
more than~$1$ and set $P\coloneqq\{p_1,\ldots,p_{\ell}\}$. Let $P'$ be the set
of children of~$z$ with degree~$1$ and set $k\coloneqq\abs{P'}$.

Note that for any~$w\in N(u)$, the definition of~$z$ ensures that
$T_w$ is a star whenever $\abs{V(T_w)} > \abs{V(T_z)}$.
The graph~$G'$ is obtained from~$G$ as follows. (An illustration is given in Figure~\ref{fig-operation}.)
For convenience, we set
$n\coloneqq n_0+n_1=\abs{V(G)}$.
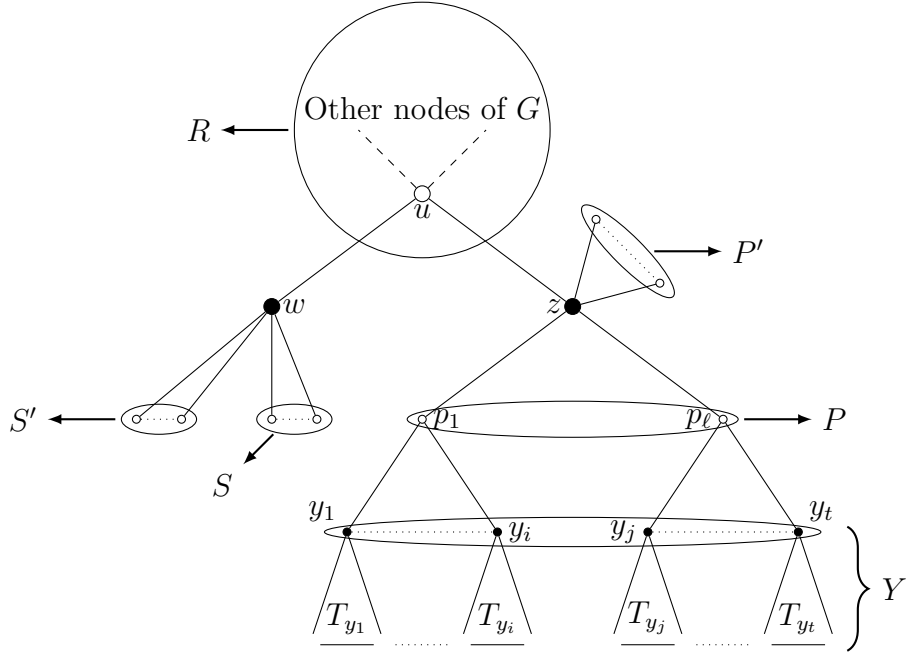
\begin{figure}[t]
\begin{center}
\begin{tikzpicture}[
level 1/.style={sibling distance=4cm},
level 4/.style={sibling distance=1cm},
vertex/.style={circle, draw=black, fill=black, inner sep=0.5pt, minimum
        size=6pt},
wvertex/.style={circle, draw=black, fill=white, inner sep=0.5pt, minimum
        size=6pt},
0-vertex/.style={circle, draw=black, fill=black, inner sep=0mm, minimum
        size=2pt},%
1-vertex/.style={circle, draw=black, fill=white, inner sep=0mm, minimum
        size=3pt},%
2-vertex/.style={circle, draw=black, fill=black, inner sep=0mm, minimum
        size=3pt}%
]
\node[wvertex](u) {}
child {node[vertex](w) {} [sibling distance=6mm]
   child {node[1-vertex](wf) {}}
   child {node[1-vertex](wl) {}}
   child {node {} edge from parent[white]}
   child {node[1-vertex](sf) {}}
   child {node[1-vertex](sl) {}}
   child {node {} edge from parent[white]}
   child {node {} edge from parent[white]}
}
child {node[vertex](z) {} [sibling distance=2cm]
   child {node[1-vertex](pf) {}
      child {node[2-vertex](yff) {}
         child {node(tyfff) {}}
         child {node(tyffl) {}}
      }
      child {node[2-vertex](yfl) {}
         child {node(tyflf) {}}
         child {node(tyfll) {}}
      }
   }
   child {node(cp) {} edge from parent[white]}
   child {node[1-vertex](pl) {}
      child {node[2-vertex](ylf) {}
         child {node(tylff) {}}
         child {node(tylfl) {}}
      }
      child {node[2-vertex](yll) {}
         child {node(tyllf) {}}
         child {node(tylll) {}}
      }
   }
}
;
   \draw[dotted] (wf)--(wl) node[midway] (cw){};
   \draw[dotted] (sf)--(sl) node[midway] (cs){};
   \draw[dotted] (yff)--(yfl);
   \draw[dotted] (ylf)--(yll);
   \draw[dotted] (tyffl)--(tyflf);
   \draw[dotted] (tylfl)--(tyllf);
   \node[left] at (z) {$z$};
   \node[right] at (w) {$w$};
   \node[below] at (u) {$u$};
   \node[right] at (pf) {$p_1$};
   \node[left] at (pl) {$p_{\ell}$};
   \node[above left] at (yff) {$y_1$};
   \node[right] at (yfl) {$y_i$};
   \node[left] at (ylf) {$y_j$};
   \node[above right] at (yll) {$y_t$};
   \draw (tyfff)--(tyffl) node[midway,above] {$T_{y_1}$};
   \draw (tyflf)--(tyfll) node[midway,above] {$T_{y_i}$};
   \draw (tylff)--(tylfl) node[midway,above] {$T_{y_j}$};
   \draw (tyllf)--(tylll) node[midway,above] {$T_{y_t}$};

   \draw (cw) ellipse (.5cm and 2mm);
   \draw[->,>=latex,thick] (cw)++(-5.8mm,0)--++(-9mm,0) node[line width=\defaultpgflinewidth,left] {$S'$};
   \draw (cs) ellipse (.5cm and 2mm);
   \draw[->,>=latex,thick] (cs)++(-2.8mm,-2mm)--++(-4mm,-4mm) node[line width=\defaultpgflinewidth,below left] {$S$};
   \draw (cp) ellipse (2.2cm and 2.4mm);
   \draw[->,>=latex,thick] (cp)++(2.28cm,0)--++(9mm,0) node[line width=\defaultpgflinewidth,right] {$P$};

   \path (yff)--(yll) node[midway,draw,ellipse,minimum width=6.6cm, minimum height=2mm] (cy){};

   \draw[thick,decorate,decoration={brace,amplitude=8pt,mirror,raise=2pt}]
   (tylll)++(.8mm,-.8mm)--++(0,1.65cm) node[midway,line width=\defaultpgflinewidth,xshift=7mm] {$Y$};
   \draw (z)--+(75:12mm) node[1-vertex](ppf){};
   \draw (z)--++(15:12mm) node[1-vertex](ppl){};
   \draw[dotted] (ppf)--(ppl) node[midway] (cpp){};
   \draw[rotate around={45:(cpp)}] (cpp) ellipse (2mm and 8.5mm);
   \draw[->,>=latex,thick] (cpp)++(3.5mm,0)--++(9mm,0) node[line width=\defaultpgflinewidth,right] {$P'$};

   \draw[dashed] (u)--++(135:12mm) node(rf){};
   \draw[dashed ](u)--+(45:12mm) node(rl){};
   \path (rf)--(rl) node[midway] (cr){};
   \node[above] at (cr) {Other nodes of $G$};
   \draw (cr) circle (17mm);
   \draw[->,>=latex,thick] (cr)++(-17.8mm,0)--++(-9mm,0) node[line width=\defaultpgflinewidth, left] {$R$};
\end{tikzpicture}
\end{center}
\caption{Figurative view of the subsets of nodes of~$G$. Recall
   that $S'\coloneqq V(T_w)\setminus \{w\}$
if $S=\varnothing$.}\label{fig:sets}
\end{figure}

\begin{enumerate}[label=(\alph*).,ref=(\alph*)]
\item For each $i\in\{1,\ldots,t\}$, the edge~$uy_i$ is added.\label{cadd}
\item For each $i\in\{1,\ldots,\ell\}$, the edge~$zp_i$ is removed and
all other edges incident to~$p_i$ but one are removed. Thus the vertices
$p_1,\ldots,p_{\ell}$ become leaves of~$G'$, each being attached
to one of the vertices~$y_1,\ldots,y_t$.\label{cdel}
\item If there exists a child~$w$ of~$u$ different from~$z$ with $\abs{V(T_w)}
\geq n/2$, then we select an arbitrary set $S \subset V(T_w)\setminus \left\{w\right\}$ of size
$\abs{V(T_w)} - \floor{n/2}$ and we set $S'\coloneqq V(T_w)\setminus(S\cup\{w\})$.
Then for each~$s \in S$, we replace the edge~$sw$ by the edge~$sz$.  \label{cw}
\item If there is no node~$w$ as in~\ref{cw}, then we let $w$ be a child of~$u$ different from~$z$
such that $\abs{V(T_w)}$ is as large as possible, and
we define~$S'$ to be~$V(T_w)\setminus\{w\}$. (Recall that $\deg_G(u)\ge2$,
hence such a child always exists.) Moreover, we set $S\coloneqq\varnothing$
for convenience.\label{dw}
\end{enumerate}
As noted earlier, if~\ref{cw} applies then $T_w$ is a
star. Moreover, if $S\neq\varnothing$, then one can see that
$W_G(w)<W_G(u)$ and hence $C_1(w;G)>C_1(u;G)$. However,
this is not a contradiction since $C_1(u;G)=\max\sst{C_1(v;G)}{v\in A_0}$
and $w\in A_1$.

Regardless of whether~\ref{cw} or~\ref{dw} applies, $\abs{S'}\le
\floor{\frac{n}{2}}-1$.  Actually, it is important to notice that, in~$G'$, no
child of~$u$ different from~$z$ has more than $\floor{n/2}-1$ children itself.
Even more, for any such child~$x$ we know that $\abs{V(T_x)}\le\floor{n/2}$.
This follows from our previous remark if $T_x$ has depth at most~$2$, and from
the fact that $\abs{V(T_x)}\le\abs{V(T_z)}$ otherwise.  Also, setting
$R\coloneqq V\setminus V(T_z) \cup V(T_w)$, we observe that for every node~$p_i\in P$
\[
\dist_G(p_i,x) = \begin{cases}
			\dist_G(u,x) - 2 & \quad\text{if $x\in V(T_{p_i})$}\\
			\dist_G(u,x) +2 & \quad\text{if $x\in R \cup V(T_w)$}\\
                            \dist_G(u,x)       & \quad\text{otherwise.}
           	           \end{cases}
\]
Therefore, $W(p_i) \le W(u) - 2 (\abs{V(T_{p_i})}-(\abs{R}+\abs{V(T_w)}))$.
Since the definition of~$u$ implies that $W(p_i) \ge W(u)$, it follows that the
size of~$V(T_{p_i})$ is at most~$\floor{n/2}$.

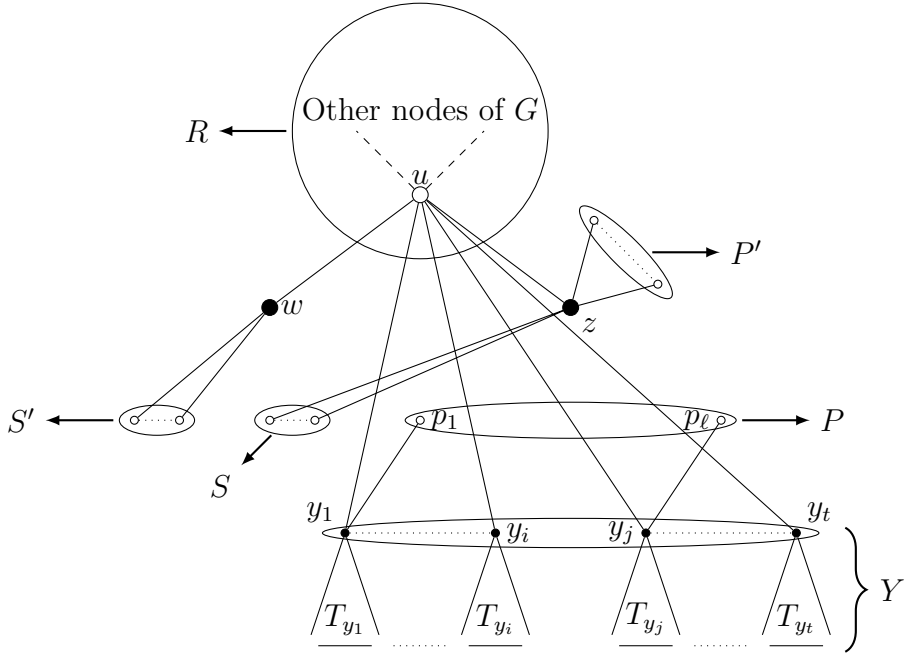
\begin{figure}%
\begin{center}
\begin{tikzpicture}[
level 1/.style={sibling distance=4cm},
level 4/.style={sibling distance=1cm},
vertex/.style={circle, draw=black, fill=black, inner sep=0.5pt, minimum
        size=6pt},
wvertex/.style={circle, draw=black, fill=white, inner sep=0.5pt, minimum
        size=6pt},
0-vertex/.style={circle, draw=black, fill=black, inner sep=0mm, minimum
        size=2pt},%
1-vertex/.style={circle, draw=black, fill=white, inner sep=0mm, minimum
        size=3pt},%
2-vertex/.style={circle, draw=black, fill=black, inner sep=0mm, minimum
        size=3pt}%
]
\node[wvertex](u) {}
child {node[vertex](w) {} [sibling distance=6mm]
   child {node[1-vertex](wf) {}}
   child {node[1-vertex](wl) {}}
   child {node {} edge from parent[white]}
   child {node[1-vertex](sf) {} edge from parent[white]}
   child {node[1-vertex](sl) {} edge from parent[white]}
   child {node {} edge from parent[white]}
   child {node {} edge from parent[white]}
}
child {node[vertex](z) {} [sibling distance=2cm]
   child {node[1-vertex](pf) {} edge from parent[white]
      child {node[2-vertex](yff) {} edge from parent[black]
         child {node(tyfff) {}}
         child {node(tyffl) {}}
      }
      child {node[2-vertex](yfl) {}
         child {node(tyflf) {} edge from parent[black]}
         child {node(tyfll) {} edge from parent[black]}
      }
   }
   child {node(cp) {} edge from parent[white]}
   child {node[1-vertex](pl) {} edge from parent[white]
      child {node[2-vertex](ylf) {} edge from parent[black]
         child {node(tylff) {}}
         child {node(tylfl) {}}
      }
      child {node[2-vertex](yll) {}
         child {node(tyllf) {} edge from parent[black]}
         child {node(tylll) {} edge from parent[black]}
      }
   }
}
;
  \draw (yff)--(u)--(yfl);
  \draw (ylf)--(u)--(yll);
   \draw (sf)--(z)--(sl);
   \draw[dotted] (wf)--(wl) node[midway] (cw){};
   \draw[dotted] (sf)--(sl) node[midway] (cs){};
   \draw[dotted] (yff)--(yfl);
   \draw[dotted] (ylf)--(yll);
   \draw[dotted] (tyffl)--(tyflf);
   \draw[dotted] (tylfl)--(tyllf);
   \node[below right] at (z) {$z$};
   \node[right] at (w) {$w$};
   \node[above] at (u) {$u$};
   \node[right] at (pf) {$p_1$};
   \node[left] at (pl) {$p_{\ell}$};
   \node[above left] at (yff) {$y_1$};
   \node[right] at (yfl) {$y_i$};
   \node[left] at (ylf) {$y_j$};
   \node[above right] at (yll) {$y_t$};
   \draw (tyfff)--(tyffl) node[midway,above] {$T_{y_1}$};
   \draw (tyflf)--(tyfll) node[midway,above] {$T_{y_i}$};
   \draw (tylff)--(tylfl) node[midway,above] {$T_{y_j}$};
   \draw (tyllf)--(tylll) node[midway,above] {$T_{y_t}$};

   \draw (cw) ellipse (.5cm and 2mm);
   \draw[->,>=latex,thick] (cw)++(-5.8mm,0)--++(-9mm,0) node[line width=\defaultpgflinewidth,left] {$S'$};
   \draw (cs) ellipse (.5cm and 2mm);
   \draw[->,>=latex,thick] (cs)++(-2.8mm,-2mm)--++(-4mm,-4mm) node[line width=\defaultpgflinewidth,below left] {$S$};
   \draw (cp) ellipse (2.2cm and 2.4mm);
   \draw[->,>=latex,thick] (cp)++(2.28cm,0)--++(9mm,0) node[line width=\defaultpgflinewidth,right] {$P$};

   \path (yff)--(yll) node[midway,draw,ellipse,minimum width=6.6cm, minimum height=2mm] (cy){};

   \draw[thick,decorate,decoration={brace,amplitude=8pt,mirror,raise=2pt}]
   (tylll)++(.8mm,-.8mm)--++(0,1.65cm) node[midway,line width=\defaultpgflinewidth,xshift=7mm] {$Y$};
   \draw (z)--+(75:12mm) node[1-vertex](ppf){};
   \draw (z)--++(15:12mm) node[1-vertex](ppl){};
   \draw[dotted] (ppf)--(ppl) node[midway] (cpp){};
   \draw[rotate around={45:(cpp)}] (cpp) ellipse (2mm and 8.5mm);
   \draw[->,>=latex,thick] (cpp)++(3.5mm,0)--++(9mm,0) node[line width=\defaultpgflinewidth,right] {$P'$};

   \draw[dashed] (u)--++(135:12mm) node(rf){};
   \draw[dashed ](u)--+(45:12mm) node(rl){};
   \path (rf)--(rl) node[midway] (cr){};
   \node[above] at (cr) {Other nodes of $G$};
   \draw (cr) circle (17mm);
   \draw[->,>=latex,thick] (cr)++(-17.8mm,0)--++(-9mm,0) node[line width=\defaultpgflinewidth, left] {$R$};
\end{tikzpicture}
\end{center}
\caption{Obtaining $G'$ from $G$. Recall that $S'\coloneqq V(T_w)\setminus \{w\}$ if $S=\varnothing$.}\label{fig-operation}%
\end{figure}

Note that $G'$ is a tree, which we see rooted at $u$, and $G$ and $G'$ have
the same node set, which we call $V$. In addition, $G$ and $G'$ have the same
bipartition $(A_0,A_1)$.
Our next task is to compare the total distance of
nodes in~$G$ and in~$G'$, that is, we compare~$W_G(x)$ and~$W_{G'}(x)$.
For readability purposes, let us set $W(x)\coloneqq W_G(x)$,
$W'(x)\coloneqq W_{G'}(x)$, and let $T_x'$ be the subtree of~$G'$ rooted
at~$x$.
We now make a few statements about~$W(x)$ and~$W'(x)$ for various nodes. We
shall often use that
\[n=\abs{V}=\abs{R}+\abs{Y}+\abs{P}+\abs{P'}+\abs{S}+\abs{S'}+2.\]
\begin{lem}\label{lem-comp}
The following hold.
\begin{enumerate}[label=\textup{(\roman*).},ref=(\roman*)]
\item If $x \in R$, then $W(x) - W'(x) = 2\abs{Y}$.\label{R}
\item If $x \in \{z\}\cup P'$, then $W'(x) \geq W(x) -
2\abs{S}$.\label{zPp}
\item If $x \in \{w\}\cup S'$, then $W'(x) = W(x) + 2\abs{S} -
2\abs{Y}$.\label{wSp}
\item If $x \in P\cup S$, then $W'(x) \geq W(x)$.\label{PS}
\item If $S\neq\varnothing$, then $W(x_1) > W(x_2)$ and $W'(x_1) > W'(x_2)$
      whenever $x_1 \in P'$ and $x_2 \in S'$.\label{PpSp}
\item If $x\in Y$, then $W'(x)\le W(x)$.\label{Y}
\item $W'(x) \geq W'(u)$ for every node $x\in Y\cup R\cup S'\cup\{w\}$.\label{RSpwY}
\end{enumerate}
\end{lem}
\begin{proof}
We prove all the statements in order.

   \medskip
	\noindent
\ref{R}.
 If $x\in R$, then the distance from $x$ to any node not in~$Y$
   is unchanged. In addition, $\dist_{G'}(x,y)=\dist_G(x,y)-2$ whenever $y\in
   Y$, hence the conclusion.

   \medskip
	\noindent
\ref{zPp}.
 If $x\in \{z\}\cup P'$, then $\dist_{G'}(x,v)\ge\dist_G(x,v)$ for each $v\in
 V\setminus S$. In addition, if $s\in S$, then $\dist_{G'}(x,s)=\dist_G(x,s)-2$, which yields the
   conclusion.

   \medskip
	\noindent
\ref{wSp}.
 It suffices to observe that if $x\in\{w\}\cup S'$, then
   \[
   \dist_{G'}(x,v)=\begin{cases}
      \dist_G(x,v)&\quad\text{if $v\in V\setminus(S\cup Y)$}\\
      \dist_G(x,v)-2&\quad\text{if $v\in Y$}\\
      \dist_G(x,v)+2&\quad\text{if $v\in S$.}
   \end{cases}
   \]

   \medskip
	\noindent
\ref{PS}.
 First note that if~$x\in P$, then the definition of~$G'$ ensures
 that $\dist_{G'}(x,v)\ge\dist_G(x,v)$ for each~$v\in V$, which implies
 that $W'(x)\ge W(x)$.

 Now let $x\in S$. Observe that if $v\in V$, then
 $\dist_{G'}(x,v)\ge\dist_{G}(x,v)-2$. In addition, if $v\in S'\cup\{w\}$, then
 $\dist_{G'}(x,v)=\dist_{G}(x,v)+2$. Consequently,
 \[
 W'(x)-W(x)\ge2\abs{S'\cup\{w\}}-2\abs{V\setminus (\{x,w\}\cup S')},
 \]
 which is non-negative since $\abs{S'\cup\{w\}}=\floor{\abs{V}/2}$ when
 $S\neq\varnothing$, and $x\notin S'\cup\{w\}$.

   \medskip
	\noindent
\ref{PpSp}.
Let $x_1\in P'$ and $x_2\in S'$. First note that every node in
$V(T_w)\setminus\{x_1\}$ is two units closer to~$x_1$ than to~$x_2$.
Similarly, every node in~$V(T_z)\setminus\{x_2\}$ is two units closer to~$x_2$
than to~$x_1$. Since, in addition, every remaining node (different from~$x_1$
and~$x_2$) is at the same distance from~$x_1$ and~$x_2$, we deduce that
\[
   W(x_1)-W(x_2)=2(\abs{S}+\abs{S'}-\abs{P}-\abs{P'}-\abs{Y}).
\]
This quantity is positive since, as $S\neq\varnothing$, we know that
$\abs{S}+\abs{S'}\ge\floor{n/2}-1$ while $\abs{P}+\abs{P'}+\abs{Y}\le
n-\abs{S}-\abs{S'}-3<\floor{n/2}-2$.

A similar analysis in~$G'$ yields that
\[
   W'(x_1)-W'(x_2)=2(\abs{S'}-\abs{S}-\abs{P'}),
\]
because every node not in $S'\cup S\cup P'\cup\{x_1,x_2\}$ is at the same
distance (in~$G'$) from~$x_1$ and~$x_2$.  Again, $\abs{S'}-\abs{S}-\abs{P'}$ is
positive since $\abs{S'}=\floor{n/2}-1$ while $\abs{P'}+\abs{S'}\le
n-\abs{S'}-3\le\floor{n/2}-2$.

\medskip
\noindent
\ref{Y}.
Let $x\in Y$. Observe that if $\dist_{G'}(x,v)>\dist_G(x,v)$, then~$v$ must be
the child of~$z$ that is an ancestor of~$x$ (that is, $v \in P$ and $x \in
V(T_v)$).  Furthermore, in this instance, the distance increases by
exactly~$2$.  As the distance from~$x$ to any node in~$R$ decreases by~$2$ (and
$\abs{R} \geq 1$), it follows that $W'(x) \leq W(x)$.

\medskip
\noindent
\ref{RSpwY}.
For readability, the proof is split intro four cases depending on whether $x\in\{w\}$,
$x\in R$, $x\in S'$ or $x\in Y$. The interested reader will notice that a similar
argument is used in all these cases, however, proceeding with cases simplifies
the verification and gives a better vision of the situation.

We start by showing that $W'(w)\ge W'(u)$. Since $\dist_{G'}(w,u)=1$, we know that
\[
\dist_{G'}(w,v)=\begin{cases}
   \dist_{G'}(u,v)-1&\quad\text{if $v\in V(T_w)\setminus S=S'\cup\{w\}$}\\
   \dist_{G'}(u,v)+1&\quad\text{otherwise.}
\end{cases}
\]
Therefore,
\begin{align*}
   W'(w)-W'(u)&=\abs{V\setminus (S'\cup\{w\})}-\abs{S'\cup\{w\}}\\
   &=\abs{V}-2(\abs{S'}+1),
\end{align*}
which is non-negative since $\abs{S'}\le\floor{n/2}-1$.

A similar reasoning applies to the nodes in~$R$.
Let $x\in R\setminus\{u\}$. Set $d\coloneqq\dist_{G'}(x,u)$ and let
$x'$ be the child of~$u$ on the unique path between~$u$ and~$x$ in~$G$.
Note that $T_{x'}'=T_{x'}$.
Since
\begin{align*}
\dist_{G'}(x,v)&=\dist_{G'}(u,v)+d\quad\quad\text{if $v\in V\setminus V(T_{x'})$}\\
\intertext{and}
\dist_{G'}(x,v)&\ge\dist_{G'}(u,v)-d\quad\quad\text{if $v\in V(T_{x'})$,}
\end{align*}
we observe that
\[
W'(x)-W'(u)\ge d\cdot\left(\abs{V\setminus V(T_{x'})}-\abs{V(T_{x'})}\right).
\]
This yields the desired inequality since, as reported earlier,
$\abs{V(T_{x'})}\le n/2$.

We now deal with the nodes in~$S'$. Let $x\in S'$. First, if
$S\neq\varnothing$, then $S'$ is composed
of precisely $\floor{n/2}-1$ nodes, which are all children of~$w$. The
definition of~$G'$ thus implies that $\dist_{G'}(x,v)\ge\dist_{G'}(u,v)$
whenever $v\neq x$, hence $W'(x)\ge W'(u)$, as asserted.
Assume now that~$S=\varnothing$.
The situation can then be dealt with
in the very same way as for the nodes in~$R$.
Indeed, in this case,
\[
   W'(x)-W'(u)\ge \dist_{G'}(x,u) \cdot\left(\abs{V\setminus V(T_w)}-\abs{V(T_w)}\right),
\]
and $T_w$ contains at most $n/2$ nodes since $S=\varnothing$.

Finally, let $x\in Y$. Similarly as before, set $d\coloneqq\dist_{G'}(x,u)$.
For every $v\in V$,
\[
\dist_{G'}(x,v)\ge\dist_{G'}(u,v) - d.
\]
Let $y_i$ be the ancestor of~$x$ among $\{y_1,\ldots,y_t\}$. If $v\notin V(T_{y_i}')$, then
\[
\dist_{G'}(x,v)=\dist_{G'}(u,v) + d.
\]
Consequently,
\[
W'(x) - W'(u) \ge d\cdot\left(\abs{V\setminus V(T_{y_i}')} - \abs{V(T_{y_i}')}\right).
\]
Now let $p_k$ be the father of~$y_i$ in~$G$. Then $V(T_{y_i}')\subseteq
V(T_{p_k})$. As reported earlier, $\abs{V(T_{p_k})}\le\floor{n/2}$, which
yields that $W'(x)-W'(u)\ge0$.
\end{proof}

The next lemma in particular bounds $C_1(u;G)$ from below.
\begin{lem}\label{lem-compY}
If $x \in Y$, then $0\le \frac{W(x)-W'(x)}{W(x)} < 2 C_1(u;G)$.
\end{lem}
\begin{proof}
Assume that $x \in V(T_{y_i})$. Lemma~\ref{lem-comp}\ref{Y} ensures that
$W'(x)\le W(x)$, thereby proving that $\frac{W(x)-W'(x)}{W(x)}$ is
non-negative.

Let $D$ be the set of those nodes whose distance to~$x$ is greater in~$G$ than
in~$G'$, that is, $D\coloneqq \sst{v \in V}{\dist_G(v,x) > \dist_{G'}(v,x)}$.
Observe that $W(x) - W'(x) \leq 2 \abs{D}$, since
$\dist_{G'}(x,v)\ge\dist_G(x,v)-2$ for every~$v\in V$.

We partition $D$ into parts $D_1,\ldots,D_m$ where
$v \in D_j$ if and only if $v \in D$ and $\dist_G(x,v) = j$.
Note that $D_1 =  \varnothing=D_2$. In addition, $D_3=\{u\}$ if
$x=y_i$ while $D_3=\varnothing$ if $x\neq y_i$.
Finally, if $x \neq y_i$, then $D_4 \subseteq\{u\}$, while
otherwise $D_4$ is contained in $A_1 \setminus \{x,z\}$. In both cases, we deduce that
$\abs{D_4} \leq  n_1 - 2$, since $n_1\ge3$.
Thus
\begin{equation}\label{eq-lbc}
W(x)-W'(x)\le2\sum_{i=3}^{m}\abs{D_i}
\end{equation}
and, since $G$ contains at least one node at distance~$2$ from~$x$,
\begin{equation}\label{eq-ldiff}
W(x) \geq 1 + 2 + \sum_{i = 3}^m i \abs{D_i}
\end{equation}

Since we assume that $C_1(u;G) \geq C_1(v;H(v;n_0,n_1))$,
it follows from Lemma~\ref{lem-lowerbound} that
$C_1(u;G)\ge\frac{n_1-1}{2(2n_1-1)}$.
Therefore,
\begin{align*}
 \frac{W(x) - W'(x)}{W(x)} - 2 C_1(u;G) & \leq
 \frac{W(x) - W'(x)}{W(x)} - \frac{n_1-1}{2n_1-1} \\
&\le\frac{2 \sum_{i=3}^{m} \abs{D_i}}{W(x)} - \frac{n_1-1}{2n_1-1} \\
&\le\frac{2(2n_1-1)\sum_{i=3}^{m} \abs{D_i}-(n_1-1)(3 + \sum_{i=3}^{m} i \abs{D_i})}{(2n_1-1) W(x) } \\
&=\frac{- 3n_1+3+\sum_{i=3}^{m} \abs{D_i}(n_1(4-i)-2+i)}{(2n_1-1)W(x)}\\
&\leq\frac{-3n_1+3 +\abs{D_3} (n_1+1) +  2\cdot \abs{D_4}}{(2n_1-1) W(x)} \\
& \leq  \frac{-3n_1+3 + (n_1+1) + 2 (n_1-2)}{(2n_1-1) W(x)} \\
&=0,
\end{align*}
where the second line follows from~\eqref{eq-lbc}, the third line from~\eqref{eq-ldiff},
and the fifth and seventh lines from our assumption that $n_1\ge3$.
\end{proof}

To complete the proof of Theorem~\ref{closenessbipartite}, what remains is to
show that $C_1(u;G') > C_1(u;G)$ which contradicts the choice of~$(G,u)$.  We
define
\[ 
      \gamma \coloneqq\sum_{u \in\{w\}\cup S'} \frac{2\abs{S}}{W(u)W'(u)} - \sum_{u \in
\{z\}\cup P'} \frac{2\abs{S}}{W(u)W'(u)}.
\]
By Lemma~\ref{lem-comp}\ref{PpSp} and the fact that $\abs{S'\cup\{w\}} \geq
\abs{P'\cup \{z\}}$ whenever $S\neq\varnothing$, we infer that $\gamma$ is
always non-negative (noticing that $\gamma=0$ if $S=\varnothing$).

Note that
\begin{align*}
      C_1(u;G') - C_1(u;G)&=\sum_{v\in V} \left[\frac{1}{W'(u)} - \frac{1}{W(u)} - \left ( \frac{1}{W'(v)} - \frac{1}{W(v)} \right )\right] \\
& =  \sum_{v \in V} \left[\frac{W(u)-W'(u)}{W(u)W'(u)} -
\frac{W(v)-W'(v)}{W(v)W'(v)}\right].
\end{align*}
For readability, set $f(v)\coloneqq\frac{W(u)-W'(u)}{W(u)W'(u)} -
\frac{W(v)-W'(v)}{W(v)W'(v)}$ and $g(v)\coloneqq\frac{1}{W(v)W'(v)}$
for each node $v\in V$.

By Lemma~\ref{lem-comp}\ref{R} and~\ref{wSp},
\[
f(v)=
\begin{cases}
2\abs{Y}(g(u)-g(v))&\quad\text{if $v\in R$}\\
2\abs{Y}(g(u)-g(v))+2\abs{S}g(v)&\quad\text{if $v\in S'\cup\{w\}$.}
\end{cases}
\]
In addition, if $v\in P\cup S$ then $W'(v)\ge W(v)$, by Lemma~\ref{lem-comp}\ref{PS}, so
$f(v)\ge2\abs{Y}g(u)$.
In total, we infer that $C_1(u;G')-C_1(u;G)$ is at least
\[
\sum_{v\in Y \cup(\{z\}\cup P')}f(v)+
\sum_{v\in R\cup S'\cup\{w\}}2\abs{Y}\cdot\left(g(u)- g(v)\right)
+2\abs{Y}\sum_{v\in P\cup S}g(u)
+\sum_{v\in S'\cup\{w\}}2\abs{S}\cdot g(v).
\]
Notice that $g(u)>\frac{1}{W'(u)}\left(\frac{1}{W(u)}-\frac{1}{W(v)}\right)$
for every node~$v\in V$.
Moreover by Lemma~\ref{lem-comp}\ref{R}, \ref{Y}, \ref{RSpwY} and Lemma~\ref{lem-compY} we know that
\begin{align*}
\sum_{v\in Y}f(v)&=2\abs{Y}\sum_{v\in Y}g(u) - \sum_{v\in Y}(W(v)-W'(v))g(v)\\
&\ge2\abs{Y}\sum_{v\in Y}g(u) - \frac{1}{W'(u)}\sum_{v\in
Y}\frac{W(v)-W'(v)}{W(v)}\\
&>2\abs{Y}\sum_{v\in Y}g(u) - \frac{\abs{Y}}{W'(u)}\cdot2C_1(u;G)\\
&>\frac{2\abs{Y}}{W'(u)}\sum_{v\in
Y}\left(\frac{1}{W(u)}-\frac{1}{W(v)}\right)-\frac{2\abs{Y}C_1(u;G)}{W'(u)}.
\end{align*}
So we infer that $C_1(u;G')-C_1(u;G)$ is greater than
\begin{multline*}
\sum_{v\in P'\cup\{z\}}f(v)+2\abs{Y}\sum_{v\in R\cup S'\cup\{w\}}\left(g(u)- g(v)\right)
+\frac{2\abs{Y}}{W'(u)}\sum_{v\in Y\cup P\cup S}\left(\frac{1}{W(u)}-\frac{1}{W(v)}\right)\\
+2\abs{S}\sum_{v\in\{w\}\cup S'}g(v)
      -2\abs{Y}\frac{C_1(u;G)}{W'(u)}.
\end{multline*}
Thanks to Lemma~\ref{lem-comp}\ref{RSpwY}, if $v\in R\cup
S'\cup\{w\}$ then
\[
   g(u)-g(v)\ge\frac{1}{W'(u)}\left(\frac{1}{W(u)}-\frac{1}{W(v)}\right).
\]
In addition, by Lemma~\ref{lem-comp}\ref{zPp} if $v\in P'\cup\{z\}$, then
   \[
   f(v)\ge2\abs{Y}g(u)-2\abs{S}g(v)
   >\frac{2\abs{Y}}{W'(u)}\left(\frac{1}{W(u)}-\frac{1}{W(v)}\right)-2\abs{S}g(v).
\]
Consequently, we deduce that
\begin{align*}
   C_1(u;G)-C_1(u;G')&>\frac{2\abs{Y}}{W'(u)}\sum_{v\in
   V}\left(\frac{1}{W(u)}-\frac{1}{W(v)}\right)-\frac{2\abs{Y}}{W'(u)}C_1(u;G)+\gamma\\
   &\ge\frac{2\abs{Y}}{W'(u)}(C_1(u;G)-C_1(u;G))\\
   &=0.
\end{align*}
This completes the proof of Theorem~\ref{closenessbipartite}.

\section{Concluding remarks and future work}
In~Figure~\ref{fig:2mode-netw:Circles-squares} we have a bipartite network~$N$
on $89$~edges with partition sizes $\abs{P_{1}}=18$ and $\abs{P_{2}}=14$ that
maximizes closeness centralization at nodes corresponding to
``Mrs.~Evelyn~Jefferson'' and to the event from ``September 16th'',
respectively.  Their closeness values are approximately equal to $0.0167$ and
$0.0192$, while their closeness centralization values are approximately equal
to $0.078$ and $0.160$, respectively. As shown in the paper, the
graphs~$H(0,18,14)$ and~$H(0,14,18)$ maximize closeness centralization among
all bipartite graphs with partition sizes~$11$ and~$28$ (regarding from which
partition we are measuring).  These graphs are depicted on
Figure~\ref{fig:Both-graphs-H11,28}.  In both graphs the maximum closeness
centralization is attained at the node labeled~$0$ with values
$C_{1}(H(0,14,18),0)\approx0.329$ and $C_{1}(H(0,11,28),0)\approx0.299$,
respectively.
\begin{figure}
\begin{centering}
\subfloat[$H(0,14,18)$]{\protect\begin{centering}
\protect
\begin{tikzpicture}[thick,scale=0.8,nvertex/.style={circle, draw=black, fill=white, inner sep=0.5pt, minimum
        size=16pt}]
        \draw (0,0) node[nvertex] (n0) {$0$};
        \draw \foreach \x [evaluate = \x as \y using int(18+\x)] in {1,2,...,13}
    {
         (20*\x:4) node[nvertex] {\y} -- (20*\x:2.5) node[nvertex] {\x}  -- (n0)
    };
        \draw \foreach \x in {14,15,...,18}
    {
         (20*\x:2.5) node[nvertex] {\x}  -- (n0)
    };
\end{tikzpicture}\quad
\protect
\par\end{centering}

}$\;$\subfloat[$H(0,18,14)$]{\protect\begin{centering}
\protect
\begin{tikzpicture}[thick,scale=0.8,nvertex/.style={circle, draw=black, fill=white, inner sep=0.5pt, minimum
        size=16pt}]
        \draw (0,0) node[nvertex] (n0) {$0$};
    \draw \foreach \x [evaluate = \x as \y using int(14+\x)] in {2,3,4,5,7,8,9,10,12,13,14}
    {
         (25.71*\x:4) node[nvertex] {\y} -- (25.71*\x:2) node[nvertex] {\x}  -- (n0)
    };
    \draw (25.71:2) node[nvertex] (n1) {$1$} -- (n0);
    \draw (n1)++(38.57:2) node[nvertex] {$29$} -- (n1);
    \draw (n1)++(12.85:2) node[nvertex] {$15$} -- (n1);
    \draw (154.26:2) node[nvertex] (n6) {$6$} -- (n0);
    \draw (n6)++(167.11:2) node[nvertex] {$30$} -- (n6);
    \draw (n6)++(141.40:2) node[nvertex] {$21$} -- (n6);
    \draw (282.81:2) node[nvertex] (n11) {$11$} -- (n0);
    \draw (n11)++(295.66:2) node[nvertex] {$31$} -- (n11);
    \draw (n11)++(269.95:2) node[nvertex] {$26$} -- (n11);
\end{tikzpicture}
\protect
\par\end{centering}

}
\par\end{centering}

\protect\caption{\label{fig:Both-graphs-H11,28}The two graphs that maximize
closeness centralization among all bipartite graphs with partition sizes~$14$
and~$18$. Note that in both cases the root is node~$0$.}
\end{figure}
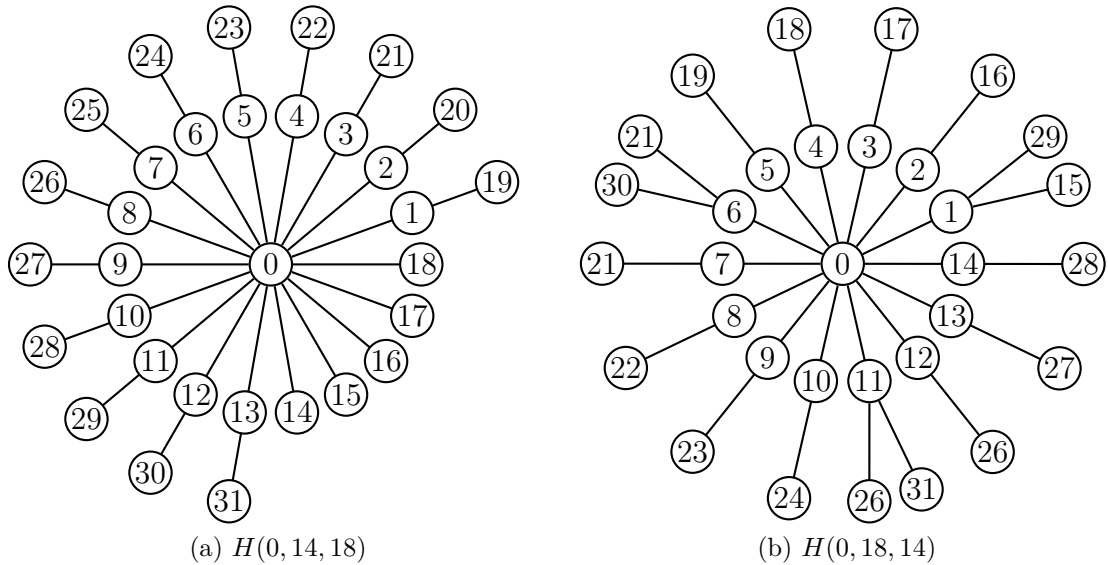

We showed that among all two-mode networks with fixed size bipartitions~$n_{0}$
and~$n_{1}$, the largest closeness centralization is achieved by a rooted tree
of depth $2$, where neighbors of the root have an equal or almost equal number
of children, namely at node~$v$ of a graph~$H(v,n_{0},n_{1})$. This confirms
a conjecture by~\citet*{ESD04} regarding the problem of maximizing closeness
centralization in two-mode data, where the number of data of each type is
fixed. A similar statement for the centrality measure of eccentricity was
recently established \citep{Krnc2015eccentricity}.  However, the same
conjecture remains open for the eigenvalue centrality~$C_e$.
\begin{conjecture}
      Let $\mathcal{B}(n_{0},n_{1})$ be the class of all bipartite
      graphs with bipartition~$P_{0}$ and~$P_{1}$, such that
      $\abs{P_{i}}=n_{i}$ for $i\in\left\{0,1\right\}$.  Then
\[
\max_{G\in\mathcal{B}(n_{0},n_{1})}\max_{v\in P_{0}}C_{e}(v,G)=C_{e}\left(v,H\left(v,n_{0},n_{1}\right)\right).
\]
\end{conjecture}
A centrality measure~$\mathcal{C}$ is said to satisfy the \emph{max-degree
property} in the family~$\mathcal{F}$ if for every graph~$G\in\mathcal{F}$ and
every node~$v\in V(G)$,
\[ \mathcal{C}_{G}(v)=\max_{u\in
V(G)}\mathcal{C}_{G}(u)\:\Longrightarrow\:\deg_{G}(v)=\max_{u\in
V(G)}\deg_{G}(u).
\]
While degree centrality trivially satisfies the
max-degree property in~$\mathcal{G}_{n}$, one can easily observe that this is
not true for closeness centrality. Still, it is interesting to observe that the
maximizing family for bipartite graphs
$H\left(v,\abs{P_{0}},\abs{P_{1}}\right)$  (or stars, for connected graphs
$\mathcal{G}_n$ in general) satisfies the max-degree property.

\bibliographystyle{abbrvnat}
\bibliography{Bibliography}{}
\end{document}